\documentclass[10pt]{article}
\usepackage{latexsym}
\usepackage{amsbsy}
\usepackage{amsfonts,mathrsfs,amssymb}
\usepackage{amsmath,amsthm}
\usepackage[latin1]{inputenc}
\usepackage{graphics,graphicx}
\usepackage{a4wide}

 \newtheorem{theorem}{Theorem}[section]
 \newtheorem{lemma}[theorem]{Lemma}
 \newtheorem{proposition}[theorem]{Proposition}

 \newtheorem{assumption}{Assumption}
 \theoremstyle{definition}
 \newtheorem{remark}[theorem]{Remark}

 \newcommand{\beq}{\begin{equation}}
 \newcommand{\eeq}{\end{equation}}

 \newcommand{\bea}{\begin{eqnarray}}
 \newcommand{\eea}{\end{eqnarray}}

 \newcommand{\beas}{\begin{eqnarray*}}
 \newcommand{\eeas}{\end{eqnarray*}}

 \newcommand{\beqs}{\begin{equation*}}
 \newcommand{\eeqs}{\end{equation*}}

 \newcommand{\bi}{\begin{itemize}}
 \newcommand{\ei}{\end{itemize}}

 \newcommand{\ben}{\begin{enumerate}}
 \newcommand{\een}{\end{enumerate}}

 \newcommand{\ba}{\begin{array}}
 \newcommand{\ea}{\end{array}}

 \newcommand{\R}{\mathbb R}
 \newcommand{\N}{\mathbb N}

 \newcommand{\cC}{\ensuremath{{\cal C}}}
 \newcommand{\cD}{\ensuremath{{\cal D}}}
 \newcommand{\cE}{\ensuremath{{\cal E}}}
 
 \newcommand{\cG}{\ensuremath{{\cal G}}}

 \newcommand{\cL}{\ensuremath{{\cal L}}}
 
 \newcommand{\cN}{\ensuremath{{\cal N}}}

 \newcommand{\cS}{\ensuremath{{\cal S}}}

 \newcommand{\eps}{\varepsilon}
 \newcommand{\vphi}{\varphi}

 \newcommand{\Id}{\mathop{\rm Id}\nolimits}

 \newcommand{\Tr}{\mathop{\rm Tr}\nolimits}
 
 \newcommand{\supp}{\mathop{\rm supp}\nolimits}

 \newcommand{\dis}[2]{\langle #1 , #2 \rangle}
 
 \newcommand{\notmid}{\mid\kern-0.5em\not\kern0.5em}

 \newcommand{\pd}{\ensuremath{\partial}}

 \newcommand{\LT}[1]{\ensuremath{{\cal L}#1}}
 \newcommand{\ILT}[1]{\ensuremath{{\cal L}^{-1}#1}}

\begin{document}

 \title{Generalized solutions for the Euler-Bernoulli model
 with Zener viscoelastic foundations and distributional
 forces\thanks{Supported by
 the Austrian Science Fund
(FWF) START program Y237 on 'Nonlinear distributional geometry',
 and the Serbian Ministry of Science Project 144016}
 }

 \author{
 G\"unther H\"ormann
 \footnote{Faculty of Mathematics, University of Vienna,
 Nordbergstr.\ 15, A-1090 Vienna, Austria,
 Electronic mail: guenther.hoermann@univie.ac.at}\\
 Sanja Konjik
 \footnote{Faculty of Sciences, Department of Mathematics and Informatics, University of Novi Sad,
 Trg D. Obradovi\'ca 4, 21000 Novi Sad, Serbia,
 Electronic mail: sanja.konjik@dmi.uns.ac.rs}\\
 Ljubica Oparnica
 \footnote{Faculty of Education, University of Novi Sad,
 Podgori\v cka 4, 25000 Sombor, Serbia,
 Electronic mail: ljubica.oparnica@pef.uns.ac.rs}
 }

 \date{}
 \maketitle

 \begin{abstract}
 We study the initial-boundary value problem for
 an Euler-Bernoulli beam model with discontinuous bending stiffness
 laying on a viscoelastic foundation and subjected to an axial
 force and an external load both of Dirac-type. The
 corresponding model equation is fourth order  partial differential equation
 and involves discontinuous and distributional coefficients
 as well as a distributional right-hand side. Moreover the
 viscoelastic foundation is of Zener type and described by
 a fractional differential equation with respect to time.
 We show how functional analytic methods for abstract variational
 problems can be applied in combination with regularization
 techniques to prove existence and uniqueness of generalized
 solutions.

 \end{abstract}

 \section{Introduction and preliminaries}
 \label{ssec:intro}

 We study existence and uniqueness of a generalized solution
 to the initial-boundary value problem
 \begin{align}
 & \pd^2_tu + Q(t,x,\pd_x)u + g = h, \label{eq:PDE} \\
 & D_t^\alpha u + u = \theta\, D_t^\alpha g + g, \label{eq:FDE}\\
 & u|_{t=0} = f_1, \quad \pd_t u|_{t=0} = f_2, \nonumber \tag{IC}\\
 & u|_{x=0} =u|_{x=1}=0, \quad
 \pd_x u|_{x=0} = \pd_x u|_{x=1}=0, \nonumber \tag{BC}
 \end{align}
 where $Q$ is a differential operator of the form
 $$
 Q u  := \pd_x^2(c(x)\pd_x^2 u) + b(x,t)\pd_x^2 u,
 $$
 $b,c,g,h,f_1$ and $f_2$ are generalized functions,
 $\theta$ a constant, $0<\theta<1$, and
 $D_t^\alpha$ denotes the left Riemann-Liouville fractional derivative
 of order $\alpha$ with respect to $t$.
 Problem (\ref{eq:PDE})-(\ref{eq:FDE}) is equivalent to
 \begin{equation} \label{eq:IntegroPDE}
 \pd^2_tu + Q(t,x,\pd_x)u + Lu = h,
 \end{equation}
 with $L$ being the (convolution) operator given by
 ($\cL$ denoting the Laplace transform)
 \begin{equation} \label{eq:operator_L}
 Lu(x,t) =
 \cL^{-1} \left(\frac{1+s^\alpha}{1+\theta s^\alpha}\right)(t)
 \ast_t u(x,t),
 \end{equation}
 with the same initial (IC) and boundary (BC) conditions
 (cf.\ Section \ref{sec:EBmodel}).

 The precise structure of the above problem is motivated by
 a model from mechanics describing the displacement of a beam
 under axial and transversal forces connected to the viscoelastic
 foundation, which we briefly discuss in Subsection
 \ref{ssec:motivation}. We then briefly introduce the theory of
 Colombeau generalized functions which forms the framework
 for our work.
 Similar problems involving distributional and generalized
 solutions to Euler-Bernoulli beam models have been studied in
 \cite{BiondiCaddemi, HoermannOparnica07, HoermannOparnica09, YavariSarkani, YavariSarkaniReddy}.
 The development of the theory in the paper is divided into
 two parts. In Section \ref{sec:abstract} we consider
 the initial-boundary value problem (\ref{eq:IntegroPDE})-(IC)-(BC)
 on the abstract level. We prove, in Theorem \ref{lemma:m-a},
 an existence result for the abstract variational problem
 corresponding to (\ref{eq:IntegroPDE})-(IC)-(BC)
 and derive energy estimates (\ref{eq:EE}) which guarantee uniqueness
 and serve as a key tool in the analysis of Colombeau
 generalized solutions. In Section \ref{sec:EBmodel},
 we first show equivalence of the system (\ref{eq:PDE})-(\ref{eq:FDE})
 with the integro-differential equation (\ref{eq:IntegroPDE}), and
 apply the results from Section \ref{sec:abstract} to the
 original problem in establishing weak solutions, if the
 coefficients are in $L^\infty$. Afterwards we allow the
 coefficients to be more irregular, set up the problem and
 show existence and uniqueness of solutions in the space of
 generalized functions.

 \subsection{The Euler-Bernoulli beam with viscoelastic foundation}
 \label{ssec:motivation}

 Consider the Euler-Bernoulli beam positioned on the
 viscoelastic foundation (cf.\ \cite{Atanackovic-book} for mechanical
 background).
 One can write the differential
 equation of the transversal motion
 \begin{equation} \label{eq:mot-trans motion}
 \frac{\pd^2}{d x^2}\left(A(x)\frac{\pd^2u}{d x^2}\right)
 + P(t) \frac{\pd^2u}{\pd x^2}
 + R(x) \frac{\pd^2u}{\pd t^2}
 + g(x,t)= h(x,t),
 \qquad x\in [0,1],\, t > 0,
 \end{equation}
 where
 \begin{itemize}
 \item $A$ denotes the bending stiffness and is given by
 $A(x) = EI_1 + H(x-x_0)EI_2$. Here, the constant
 $E$ is the modulus of elasticity,
 $I_1$, $I_2$, $I_1\neq I_2$, are the moments of
 inertia that correspond to the two parts of the beam,
 and $H$ is the Heaviside jump function;
 \item $R$ denotes the line density (i.e., mass per length)
 of the material and is of the form
 $ R(x)= R_0 + H(x-x_0)(R_1-R_2)$;
 \item $P(t)$ is the axial force, and is assumed to be of
 the form $P(t)=P_0 + P_1\delta(t-t_1)$, $P_0,P_1>0$;
 \item $g=g(x,t)$ denotes the force terms coming from
 the foundation;
 \item $u=u(x,t)$ denotes the displacement;
 \item $h=h(x,t)$ is the prescribed external load (e.g. when
 describing moving load it is of the form
 $h(x,t)=H_0\delta(x-ct)$, $H_0$ and $c$ are constants).
 \end{itemize}
 Since the beam is connected to the viscoelastic
 foundation there is a constitutive equation describing relation
 between the force of foundation and the displacement of the beam.
 We use the Zener generalized model given by
 \begin{equation} \label{eq:mot-const eq}
 D_t^\alpha u(x,t) + u(x,t) = \theta\, D_t^\alpha g(x,t) + g(x,t),
 \end{equation}
 where $0<\theta<1$ and $D_t^\alpha$ denotes the left
 Riemann-Liouville fractional derivative of order $\alpha$
 with respect to $t$, defined by
 $$
 D_t^\alpha u (t) = \frac{1}{\Gamma (1-\alpha)} \frac{d}{dt}
 \int_0^t \frac{u(\tau)}{(t-\tau)^\alpha} \,d\tau.
 $$

 System (\ref{eq:mot-trans motion})-(\ref{eq:mot-const eq}) is supplied
 with initial conditions
 $$
 u(x,0) = f_1(x), \qquad \pd_t u (x,0) = f_2(x),
 $$
 where $f_1$ and $f_2$ are the initial displacement and the initial
 velocity. If $f_1(x)=f_2(x)=0$ the only solution to
 (\ref{eq:mot-trans motion})-(\ref{eq:mot-const eq})
 should be $u\equiv g\equiv 0$.
 Also, the beam is considered to be
 fixed at both ends, hence boundary conditions take the form
 $$
 u(0,t) = u(1,t) = 0, \qquad \pd_x u (0,t)=\pd_x u (1,t) = 0.
 $$

 By a change of variables $t\mapsto \tau$ via
 $t(\tau) = \sqrt{R(x)}\tau$ the problem
 (\ref{eq:mot-trans motion})-(\ref{eq:mot-const eq}) is
 transformed into the standard form given in
 (\ref{eq:PDE})-(\ref{eq:FDE}). The function $c$ in
 (\ref{eq:PDE}) equals $A$ and therefore is of Heaviside
 type, and the function $b$ is then given by
 $b(x,t)= P(R(x)t)$ and its regularity properties
 depend on the assumptions on $P$ and $R$.

 As we shall see in Section \ref{sec:EBmodel}, standard functional
 analytic techniques reach as far as the following:
 boundedness of $b$ together with sufficient (spatial Sobolev)
 regularity of the initial values $f_1, f_2$ ensure existence
 of a unique solution $u\in L^2(0,T; H^2_0((0,1)))$ (in fact
 $u\in L^2(0,T; H^2_0((0,1)))$) to (\ref{eq:IntegroPDE}) with
 (IC) and (BC). However, the prominent case
 $b = p_0 + p_1 \delta (t-t_1)$ is clearly not covered by such
 a result, so in order to allow for these stronger singularities
 one needs to go beyond distributional solutions.

 \subsection{Basic spaces of generalized functions}
 \label{ssec:Colombeau}

 We shall set up and solve Equation (\ref{eq:IntegroPDE}) subject
 to the initial and boundary conditions (IC) and (BC)
 in an appropriate space of Colombeau generalized functions on the
 domain $X_T := (0,1) \times (0,T)$ (with $T > 0$) 
 as introduced in \cite{BO:92} and applied later on, e.g., also in
 \cite{GH:04,HoermannOparnica09}. As a few standard references for
 the general background concerning Colombeau algebras on arbitrary
 open subsets of $\R^d$ or on manifolds we may mention \cite{CCR, c1, book, MOBook}.

 We review the basic notions and facts about the kind of
 generalized functions we will employ below: we start with
 regularizing families $(u_{\eps})_{\eps\in (0,1]}$ of smooth functions
 $u_{\eps}\in  H^{\infty}(X_T)$ (space of smooth functions on
 $X_T$ all of whose derivatives belong to $L^2$). We will often
 write $(u_{\eps})_{\eps}$ to mean $(u_{\eps})_{\eps\in (0,1]}$.
 We consider the following subalgebras:

 {\it Moderate families}, denoted by $\cE_{M,H^{\infty}(X_T)} $, are
 defined by the property
 $$
 \forall\,\alpha\in\N_0^n, \exists\,p\geq 0:
 \|\pd^{\alpha}u_{\eps}\|_{L^2(X_T)}= O(\eps^{-p}), \quad \text{ as }
 \eps\to 0.
 $$

 {\it Null families}, denoted by $\cN_{H^{\infty}(X_T)}$,
 are the families in $\cE_{M,H^{\infty}(X_T)}$ satisfying
 $$
 \forall\,q\geq 0: \|u_{\eps}\|_{L^2(X_T)} = O(\eps^q) \quad \text{ as }
 \eps\to 0.
 $$
 Thus moderateness requires $L^2$ estimates with at most
 polynomial divergence as $\eps\to 0$, together with all
 derivatives,  while null families vanish very rapidly as $\eps \to 0$.
 We remark that null families in fact have all derivatives  satisfy
 estimates of the same kind (cf. \cite[Proposition 3.4(ii)]{garetto_duals}).
 Thus null families form a differential ideal in the collection
 of moderate families and we may define the {\it Colombeau algebra}
 as the factor algebra
 $$
 \cG_{H^{\infty}(X_T)} =
 \cE_{M, H^{\infty}(X_T)}/\cN_{H^{\infty}(X_T)}.
 $$

 A typical notation for the
 equivalence classes  in
 $\cG_{H^{\infty}(X_T)}$ with representative $(u_\eps)_\eps$ will be
 $[(u_\eps)_\eps]$.  Finally, the algebra $\cG_{H^{\infty}((0,1))}$
 of generalized functions on the interval $(0,1)$ is defined
 similarly and every element can be considered to be a member
 of $\cG_{H^{\infty}(X_T)}$ as well.

 We briefly recall a few technical remarks from \cite[Subsection 1.2]{HoermannOparnica09}:

 If  $(u_{\eps})_{\eps}$ belongs to $\cE_{M,H^{\infty}(X_T)}$
 we have smoothness up to the boundary for every $u_\eps$, i.e.\
 $u_{\eps}\in C^{\infty}([0,1] \times [0,T])$  (which follows from
 Sobolev space properties on the Lipschitz domain $X_T$; cf.\ \cite{AF:03})
 and therefore
 the restriction $u |_{t=0}$  of a generalized function
 $u \in \cG_{H^{\infty}(X_T)}$ to $t=0$ is well-defined by
 $u_{\eps}(\cdot,0)\in \cE_{M,H^{\infty}((0,1))}$.

 If  $v \in \cG_{H^{\infty}((0,1))}$ and in addition
 we have for some representative $(v_\eps)_\eps$ of $v$ that
 $v_\eps \in H_0^{2}((0,1))$, then $v_\eps(0) = v_\eps(1) = 0$
 and $\pd_x v_\eps (0) = \pd_x v_\eps(1) = 0$. In particular,
 $$
 v(0) = v(1) = 0 \quad\text{and}\quad \pd_x v(0) = \pd_x v(1) = 0
 $$
 holds in the sense of generalized numbers.
%


Note that $L^2$-estimates for parametrized families
 $u_\eps \in H^\infty(X_T)$ always yield similar $L^\infty$-estimates
  concerning $\eps$-asymptotics
 (since $H^\infty(X_T) \subset C^\infty(\overline{X_T})
 \subset W^{\infty,\infty}(X_T)$).


 The space $H^{-\infty}(\R^d)$, i.e.\ distributions of finite order,
 is embedded (as a linear space) into $\cG_{H^{\infty}(\R^d)}$ by
 convolution regularization (cf.\ \cite{BO:92}). This embedding
 renders $H^\infty(\R^d)$ a subalgebra of $\cG_{H^{\infty}(\R^d)}$.

 Certain generalized functions possess distribution aspects, namely
 we call $u =[(u_{\eps})_{\eps}]\in \cG_{H^{\infty}}$
 {\it associated with the distribution} $w\in\cD'$, notation
 $u \approx w$, if for some (hence any) representative $(u_{\eps})_{\eps}$
 of $u$ we have $u_{\eps}\to w$ in $\cD'$, as $\eps\to 0$.



 \section{Preparations: An abstract evolution problem in variational form and the convolution-type operator $L$}
 \label{sec:abstract}

 In this section we study an abstract background of equation
 (\ref{eq:IntegroPDE}) subject to the initial and boundary
 conditions (IC) and (BC) in terms of bilinear forms on
 arbitrary Hilbert spaces.
 First we shall repeat standard results and then extend them to a
 wider class of problems.
 We shall show existence of a unique solution, derive energy
 estimates, and analyze the particular form of the operator
 $L$  appearing in (\ref{eq:IntegroPDE}).

 Let $V$ and $H$ be two separable Hilbert spaces,
 where $V$ is densely embedded into $H$. We shall denote the norms
 in $V$ and $H$ by $\|\cdot\|_V$ and $\|\cdot\|_H$ respectively.
 If $V'$ denotes the dual of $V$, then $V \subset H \subset V'$
 forms a Gelfand triple.
 In the sequel we shall also make use of the Hilbert spaces
 $E_V:=L^2(0,T;V)$ with the norm
 $\|u\|_{E_V}:=(\int_0^T \|u(t)\|_V^2\,dt)^{1/2}$, and
 $E_H:=L^2(0,T;H)$ with the norm
 $\|u\|_{E_H}:=(\int_0^T \|u(t)\|_H^2\,dt)^{1/2}$.
 Since, $\|v\|_H\leq \cC\cdot\|v\|_V$, $v\in V$,
 (without loss of generality we may assume that $\cC=1$),
 it follows that $\|u\|_{E_H}\leq \|u\|_{E_V}$, $u\in E_V$,
 and $E_V\subset E_H$.
 The bilinear forms we shall deal with will be of the following
 type:

 \begin{assumption} \label{Ass1}
 Let $a(t,\cdot,\cdot)$, $a_0(t,\cdot,\cdot)$ and
 $a_1(t,\cdot,\cdot)$, $t\in [0,T]$,
 be (parametrized) families of continuous bilinear forms on $V$
 with
 $$
 a(t,u,v)=a_0(t,u,v)+ a_1(t,u,v) \qquad \forall\, u, v \in V,
 $$
 such that the 'principal part' $a_0$ and the remainder $a_1$
 satisfy the following conditions:
 \begin{itemize}
 \item[(i)] $t \mapsto a_0(t,u,v)$ is continuously
 differentiable $[0,T] \to \R$, for all $u,v \in V$;
 \item[(ii)] $a_0$ is symmetric, i.e., $a_0(t,u,v)=
 a_0(t,v,u)$, for all $u,v\in V$;
 \item[(iii)] there exist real constants $\lambda,\mu >0$ such that
 \beq \label{eq:coercivity}
 a_0(t,u,u) \geq \mu \|u\|_V^2 - \lambda \|u\|_H^2,
 \qquad \forall\, u\in V,\, \forall\, t\in [0,T];
 \eeq
 \item[(iv)] $t \mapsto a_1(t,u,v)$ is continuous $[0,T] \to \R$,
 for all $u,v\in V$;
 \item[(v)] there exists $C_1 \geq 0$ such that for all $t\in [0,T]$ and
 $u,v\in V$, $|a_1(t,u,v)|\leq C_1 \|u\|_V\, \|v\|_H$.
 \end{itemize}
 \end{assumption}

 It follows from condition (i) that there exist nonnegative
 constants $C_0$ and $C_0'$ such that for all $t\in [0,T]$ and
 $u,v\in V$,
 \beq \label{i_cons}
 |a_0(t,u,v)| \leq C_0 \|u\|_V\,\|v\|_V
 \quad \mbox{ and } \quad
 |a'_0(t,u,v)| \leq C_0' \|u\|_V \,\|v\|_V,
 \eeq
 where $a'_0(t,u,v):=\frac{d}{dt} a_0(t,u,v)$.

 It is shown in \cite[Ch.\ XVIII, p.\ 558, Th.\ 1]{DautrayLions-vol5}
 (see also \cite[Ch.\ III, Sec.\ 8]{LionsMagenes})
 that the above conditions guarantee unique solvability of the
 abstract variational problem in the following sense:

 \begin{theorem} \label{th:avp}
 Let $a(t,\cdot,\cdot)$, $t\in[0,T]$, satisfy Assumption
 \ref{Ass1}.
 Let $u_0\in V$, $u_1\in H$ and $f\in E_H$.
 Then there exists a unique $u \in E_V$ satisfying
 the regularity conditions
 \beq \label{eq:avp-reg}
 u'= \frac{du}{dt}\in E_V
 \quad \mbox{ and } \quad
 u''=\frac{d^2u}{dt^2}\in L^2(0,T;V')
 \eeq
 (here time derivatives should be understood in distributional
 sense),
 and solving the abstract initial value problem
 \begin{align}
 &\dis{u''(t)}{v} + a(t,u(t),v)=\dis{f(t)}{v},
 \qquad \forall\, v\in V, \,
 \mbox{ for a.e. } t \in (0,T) \label{eq:avp}\\
 & u(0)=u_0,\qquad u'(0)=u_1. \label{eq:avp-ic}
 \end{align}
 (Note that (\ref{eq:avp-reg}) implies that $u \in C([0,T],V)$ and
 $u' \in C([0,T],V')$. Hence it makes sense to evaluate $u(0)\in V$
 and $u'(0) \in V'$ and (\ref{eq:avp-ic}) claims that these equal
 $u_0$ and $u_1$, respectively.)
 \end{theorem}

 \begin{remark} \label{rem:distr-intrpr}
 The precise meaning of (\ref{eq:avp}) is the following:
 $\forall\, \vphi\in\cD((0,T))$,
 $$
 \dis{\dis{u''(t)}{v}}{\vphi}_{(\cD',\cD)}
 + \dis{a(t,u(t),v)}{\vphi}_{(\cD',\cD)}
 = \dis{\dis{f(t)}{v}}{\vphi}_{(\cD',\cD)},
 $$
 or equivalently,
 $$
 \int_0^T \dis{u(t)}{v}\vphi''(t)\,dt
 + \int_0^T a(t,u(t),v)\vphi(t)\,dt
 = \int_0^T \dis{f(t)}{v}\vphi(t)\,dt.
 $$
 \end{remark}

 The proof of this theorem proceeds by showing that $u$ satisfies
 a priori (energy) estimates which immediately imply uniqueness of
 the solution, and then using the Galerkin approximation method to
 prove existence of a solution. An explicit form of the energy
 estimate for the abstract variational problem
 (\ref{eq:avp-reg})-(\ref{eq:avp-ic})
 with precise dependence of all constants is derived in
 \cite[Prop.\ 1.3]{HoermannOparnica09} in the form
 \beq \label{eq:ee-ThmP1}
 \|u(t)\|^2_V + \|u'(t)\|^2_H \leq
 \left(
 D_T\|u_0\|_V^2 + \|u_1\|_H^2 + \int_0^t \|f(\tau)\|_H^2\,d\tau
 \right) \cdot e^{t\cdot F_T},
 \eeq
 where $D_T:=\frac{C_0 + \lambda (1+ T)}{\min\{1,\mu\}}$ and
 $F_T := \max \{\frac{C_0' +C_1}{\min\{1,\mu\}},
 \frac{C_1 +T+ 2)}{\min\{1,\mu\}}\}$.\\

 \subsection{Existence of a solution to the abstract variational problem}
 \label{ssec:existence}

 We shall now prove a similar result for a slightly modified
 abstract variational problem, which is to encompass our problem
 (\ref{eq:IntegroPDE}). Here in addition to the bilinear forms we
 shall consider "causal" operators $L:L^2(0,T_1;H)\to L^2(0,T_1;H)$,
 $\forall\,T_1<T$, which satisfy the estimate: $\exists\, C_L>0$
 such that
 \beq \label{eq:L-estimate}
 \|Lu\|_{L^2(0,T_1;H)} \leq C_L \|u\|_{L^2(0,T_1;H)},
 \eeq
 where $C_L$ is independent of $T_1$.

 \begin{lemma} \label{lemma:m-a}
 Let $a(t,\cdot,\cdot)$, $t\in[0,T]$, satisfy Assumption
 \ref{Ass1}.
 Let $f_1\in V$, $f_2\in H$ and $h\in E_H$.
 Let $L:E_H\to E_H$ satisfy (\ref{eq:L-estimate}).
 Then there exists a $u\in E_V$ satisfying the regularity conditions
 $$
 u'=\frac{du}{dt} \in E_V \quad \mbox{ and } \quad
 u''=\frac{d^2 u}{dt^2} \in L^2(0,T;V')
 $$
 and solving the abstract initial value problem
 \begin{align}
 &\dis{u''(t)}{v} + a(t,u(t),v) + \dis{Lu(t)}{v}= \dis{h(t)}{v},
 \qquad \forall\,  v\in V,\,
 \mbox{ for a.e. } t \in (0,T), \label{eq:avp-L}\\
 & u(0)=f_1, \qquad u'(0) = f_2. \label{eq:avp-L-ic}
 \end{align}
 Moreover, we have $u\in\cC([0,T];V)$ and $u'\in\cC([0,T];H)$.
 \end{lemma}

 Here we give a proof based on an iterative procedure and employing
 Theorem \ref{th:avp} and the energy estimate (\ref{eq:ee-ThmP1})
 in each step. Notice that the precise meaning of (\ref{eq:avp-L})
 (in distributional sense) is explained in Remark
 \ref{rem:distr-intrpr}.

 \begin{proof}
 Let $u_0\in E_H$ be arbitrarily chosen and consider the initial
 value problem for $u$ in the sense of Remark \ref{rem:distr-intrpr}
 \begin{align}
 & \dis{u''(t)}{v}+ a(t,u(t),v)+ \dis{Lu_0(t)}{v}= \dis{h(t)}{v},
 \qquad \forall\,  v\in V,\,
 \mbox{ for a.e. } t \in (0,T), \label{eq:sol u1}\\
 & u(0) = f_1,\quad u'(0) = f_2. \nonumber
 \end{align}
 By Theorem \ref{th:avp} there exists a unique
 $u_1\in E_V$ satisfying $u_1'\in E_V$,
 $u_1''\in L^2(0,T;V')$, and solving (\ref{eq:sol u1}).
 Consider now (\ref{eq:sol u1}) with $Lu_1$ instead of
 $Lu_0$. As above, by Theorem \ref{th:avp}, one obtains
 a unique solution $u_2\in E_V$ with $u_2'\in E_V$
 and $u_2''\in L^2(0,T;V')$.
 Repeating this procedure we obtain a sequence of functions
 $\{u_k\}_{k\in\N}\in E_V$,
 satisfying $u_k'\in E_V$, $u_k''\in L^2(0,T;V')$, and
 solving the following problems: for each $k\in\N$,
 \begin{align*}
 & \dis{u_k''(t)}{v} + a(t,u_k(t),v) + \dis{Lu_{k-1}(t)}{v}= \dis{h(t)}{v},
 \qquad \forall\,  v\in V,\,
 \mbox{ for a.e. } t \in (0,T), \\
 & u_k(0) = f_1,\quad u_k'(0) = f_2.
 \end{align*}
 Also, for all $k\in\N$, $u_k$ satisfies the energy estimate of
 type (\ref{eq:ee-ThmP1}):
 $$
 \|u_k(t)\|_V^2 + \|u_k'(t)\|_H^2
 \leq \left(
 D_T \,\|f_1\|_V^2 + \|f_2\|_H^2
 + \int_0^t\|(h-Lu_{k-1})(\tau)\|_H^2 \,d\tau
 \right) \cdot e^{t\cdot F_T},
 $$
 where the constants $D_T$ and $F_T$ are independent of $k$. We claim that $\{u_k\}_{k\in\N}$
converges in $E_V$.
 To see this, we first note that $u_l-u_k$ solves
 \begin{align*}
 & \dis{(u_l-u_k)''(t)}{v} + a(t,(u_l-u_k)(t),v) +
 \dis{L(u_{l-1}-u_{k-1})(t)}{v}= 0, \quad \forall\,v\in V,\,
 \mbox{ for a.e. } t \in (0,T), \\
 & (u_l-u_k)(0) = 0, \quad (u_l-u_k)'(0) = 0,
 \end{align*}
 and $u_l-u_k\in E_V$, with $u_l'-u_k'\in E_V$ and
 $u_l''-u_k''\in L^2(0,T;V')$. Moreover, the corresponding energy
 estimate is of the form
 \begin{equation} \label{EE1}
 \|(u_l-u_k)(t)\|_V^2 + \|(u_l-u_k)'(t)\|_H^2 \leq
 e^{t\cdot F_T} \cdot
 \int_0^t \|L(u_{l-1}-u_{k-1})(\tau)\|_H^2\,d\tau.
 \end{equation}
 Thus,
 $$
 \|(u_l-u_k)(t)\|_V^2
 \leq e^{T\cdot F_T} \cdot \int_0^T
 \|L(u_{l-1}-u_{k-1})(\tau)\|_H^2\,d\tau
 = e^{T\cdot F_T} \cdot \|L(u_{l-1}-u_{k-1})\|_{E_H}^2.
 $$
 Integrating from $0$ to $T$ and using assumption
 (\ref{eq:L-estimate}) on $L$, one obtains
 \begin{equation} \label{uk-ul_estimates}
 \|u_l-u_k\|_{E_V}
 \leq \gamma_T \|u_{l-1}-u_{k-1}\|_{E_H},
 \end{equation}
 where $\gamma_T:=C_L \sqrt{T} e^{\frac{T\cdot F_T}{2}}$.
 Taking now $l=k+1$ in (\ref{uk-ul_estimates}) successively, yields
 $$
 \|u_{k+1}-u_k\|_{E_V} \leq \gamma_T^k \|u_{1}-u_{0}\|_{E_H} \leq
 \gamma_T^k \|u_{1}-u_{0}\|_{E_V},
 $$
 and hence
 $$
 \|u_l-u_k\|_{E_V} \leq \|u_l-u_{l-1}\|_{E_V}+\ldots
 +\|u_{k+1}-u_{k}\|_{E_V}
 \leq \sum_{i=l-1}^{k} \gamma_T^i \|u_{1}-u_{0}\|_{E_V}.
 $$

 We may choose $T_1<T$ such that $\gamma_{T_1}< 1$, hence
 $\sum_{i=0}^\infty \gamma_{T_1}^i$ converges.
 Note that $t\mapsto\gamma_t$ is increasing. By abuse of notation
 we denote $L^2(0,T_1;V)$ again by $E_V$.
 This further implies that $\{u_k\}_{k\in\N}$ is a Cauchy sequence
 and hence convergent in $E_V$, say
 $u:= \lim_{k\to\infty} u_k$. Similarly,
 one can show convergence of $u_k'$ in $E_V$, i.e., existence of
 $v:= \lim_{k\to\infty} u_k' \in E_V$.
 In the distributional setting $\lim_{k\to\infty} u_k' = u'$,
 and therefore $u'=v\in E_V$
 (cf.\ \cite[Ch.\ XVIII, p.\ 473, Prop.\ 6]{DautrayLions-vol5}).

 We also have to show that $u$ solves equation (\ref{eq:avp-L}).
 Let $\vphi\in \cD((0,T))$. Then
 \begin{align*}
 \dis{\dis{h(t)}{v}}{\vphi} & =
 \dis{\dis{u_k''(t)}{v}}{\vphi} + \dis{a(t,u_k(t),v)}{\vphi}
 + \dis{\dis{Lu_{k-1}(t)}{v}}{\vphi} \\
 & = \dis{\dis{u_k}{v}}{\vphi''} + \dis{a(t,u_k(t),v)}{\vphi}
 + \dis{\dis{Lu_{k-1}(t)}{v}}{\vphi} \\
 & \to \dis{\dis{u}{v}}{\vphi''} + \dis{a(t,u(t),v)}{\vphi}
 + \dis{\dis{Lu(t)}{v}}{\vphi} \\
 & = \dis{\dis{u''(t)}{v}}{\vphi} + \dis{a(t,u(t),v)}{\vphi}
 + \dis{\dis{Lu(t)}{v}}{\vphi}.
 \end{align*}
 Here we used that $\vphi''\in \cD((0,T))$.
 Therefore $u$ solves (\ref{eq:avp-L}) on the time interval
 $[0,T_1]$. The initial conditions are satisfied by construction
 of $u$.

 It remains to extend this result on existence of a solution to
 the whole interval $[0,T]$.

 Since $T_1$ is independent on the initial conditions, if $T>T_1$
 one needs at most $\frac{T}{T_1}$ steps to reach
 convergence in $E_V$. In fact, one has to show
 regularity at the end point $T_1$ of the interval $[0,T_1]$
 on which the solution exists, i.e.,
 $$
 u(T_1)\in V \quad \mbox{ and } \quad
 u'(T_1)\in H.
 $$
 To see this, it suffices to show that $u_k\to u$ in
 $Y_V:=\cC([0,T_1];V)$ and $u_k'\to u'$ in $Y_H:=\cC([0,T_1];H)$.
 From (\ref{EE1}) and assumption (\ref{eq:L-estimate}) on $L$
 we obtain
 $$
 \|(u_l-u_k)(t)\|_V^2
 \leq e^{T_1\cdot F_{T_1}} C_L^2 \int_0^{T_1} \|(u_l-u_k)(\tau)\|_H^2 \,d\tau
 \leq e^{T_1\cdot F_{T_1}} C_L^2 \int_0^{T_1} \|(u_l-u_k)(\tau)\|_V^2 \,d\tau.
 $$
 Taking first the square root and then the supremum over all
 $t\in[0,T]$ yields
 $$
 \|u_l-u_k\|_{Y_V} \leq \gamma_{T_1} \|u_l-u_k\|_{Y_V}.
 $$
 Since $\gamma_{T_1}<1$ this implies that $\{u_k\}_{k\in\N}$
 is a Cauchy sequence in $Y_V$.
 Similarly,
 $$
 \|(u_l-u_k)'(t)\|_H^2 \leq e^{T_1\cdot F_{T_1}} C_L^2
 \int_0^{T_1} \|(u_l-u_k)(\tau)\|_V^2 \,d\tau,
 $$
 which upon taking the supremum gives
 $$
 \|(u_l-u_k)'\|_{Y_H} \leq \gamma_{T_1}
 \|u_l-u_k\|_{Y_V},
 $$
 thus $u'_k\to u'$ in $Y_H$ (due to the already established
 convergence of $u_k$ in $Y_V$),
 and $u'(T_1)\in H$. This proves the claim.
 \end{proof}

 \subsection{Energy estimates}
 \label{ssec:ee}

 In Section \ref{sec:EBmodel} we shall need
 a priori (energy) estimate for problem (\ref{eq:IntegroPDE}).
 In fact, for the verification of moderateness in the Colombeau
 setting it will be crucial to know all constants in the energy
 estimate precisely. Therefore, we shall now derive it.

 \begin{proposition} \label{prop:EnergyEstimates}
 Under the assumptions of Lemma \ref{lemma:m-a},
 let $u$ be a solution to the abstract
 variational problem (\ref{eq:avp-L})-(\ref{eq:avp-L-ic}).
 Then, for each $t\in [0,T]$,
 \begin{equation}\label{eq:EE}
  \|u(t)\|^2_V + \|u'(t)\|^2_H \leq
  \left(
 D_T\|f_1\|_V^2 + \frac{1}{\nu} \left(
 \|f_2\|_H^2 + \int_0^{t} \|h(\tau)\|^2_H \,d\tau
 \right) \right)
 \cdot e^{t\cdot F_T},
 \end{equation}
 where $\nu:=\min\{1,\mu\}$,
 $D_T:=\frac{C_0 + \lambda (1+ T)}{\nu}$ and
 $F_T := \max \{\frac{C_0' +C_1+ C_L}{\nu},
 \frac{C_1 +2+ \lambda (1+ T)}{\nu}\}$.
 \end{proposition}

 \begin{proof}
 Setting $v:= u'(t)$ in (\ref{eq:avp-L})
 we obtain (as an equality of integrable functions with
 respect to $t$)
 $$
 \dis{u''(t)}{u'(t)} + a(t,u(t),u'(t)) + \dis{Lu(t)}{u'(t)}
 = \dis{h(t)}{u'(t)}.
 $$
 Since $a(t,u,v)=a_0(t,u,v)+a_1(t,u,v)$ and
 $\dis{u''(t)}{u'(t)}= \frac{1}{2}\frac{d}{dt}\dis{u'(t)}{u'(t)}
 =\frac{1}{2}\frac{d}{dt}\|u'(t)\|^2_H$, we have
 $$
 \frac{d}{dt}\|u'(t)\|^2_H = - 2 a_0(t,u(t),u'(t))
 - 2a_1(t,u(t), u'(t)) - 2 \dis{Lu(t)}{u'(t)}
 + 2 \dis{h(t)}{u'(t)}.
 $$
 Integration from $0$ to $t_1$, for arbitrary $0< t_1 \leq T$,
 gives
 \beas
 \|u'(t_1)\|^2_H - \|f_2\|^2_H
 &=&
 - 2 \int_0^{t_1} a_0(t,u(t),u'(t))\,dt
 -2 \int_0^{t_1} a_1(t,u(t), u'(t))\,dt \\
 && \quad
 - 2 \int_0^{t_1}\dis{Lu(t)}{u'(t)}\,dt
 + 2 \int_0^{t_1} \dis{h(t)}{u'(t)}\,dt.
 \eeas
 Note that $\frac{d}{dt}a_0(t,u(t),u(t))= a_0'(t,u(t),u(t))
 +a_0(t,u'(t),u(t))+a_0(t,u(t),u'(t))$ and hence,
 by Assumption \ref{Ass1} (ii),
 $2a_0(t,u'(t),u(t))=\frac{d}{dt}a_0(t,u(t),u(t))-a_0'(t,u(t),u(t))$.
 This yields
 \begin{align}
 & LHS := \|u'(t_1)\|^2_H  + a_0(t_1, u(t_1), u(t_1))
 = \|f_2\|^2_H + a_0(0,u(0),u(0))
 - \int_0^{t_1} a_0'(t,u(t),u(t))\,dt \nonumber \\
 & \qquad
 - 2\int_0^{t_1} a_1(t,u(t), u'(t)) \,dt
 - 2\int_0^{t_1} \dis{Lu(t)}{u'(t)}\,dt
 + 2\int_0^{t_1} \dis{h(t)}{u'(t)}\,dt
 = : RHS. \label{eq:srednja}
 \end{align}
 Further, by (\ref{i_cons}), Assumption \ref{Ass1} (v), the
 Cauchy-Schwartz inequality, the inequality $2ab\leq a^2+b^2$,
 and the assumption (\ref{eq:L-estimate}) on $L$ we have
 \begin {align*}
 |RHS|
 & \leq \|f_2\|^2_H + C_0\|u(0)\|^2_V
 + C_0' \int_0^{t_1} \|u(t)\|^2_V \,dt \\
 & \quad
 + 2 C_1 \int_0^{t_1} \|u(t)\|_V \| u'(t)\|_H \,dt
 + 2 \int_0^{t_1} \|Lu(t)\|_H \|u'(t)\|_H \,dt
 + 2 \int_0^{t_1} \|h(t)\|_H \|u'(t)\|_H \,dt \\
 & \leq \|f_2\|^2_H + C_0\|f_1\|^2_V
 + (C_0' +C_1) \int_0^{t_1} \| u(t)\|^2_V \,dt \\
 & \quad
 + (C_1 +2) \int_0^{t_1} \| u'(t)\|^2_H \,dt
 + \|Lu\|^2_{L^2(0,t_1;H)}
 + \int_0^{t_1} \|h(t)\|^2_H \,dt \\
 & \leq \|f_2\|^2_H + C_0\|f_1\|^2_V
 + (C_0' +C_1+ C_L) \int_0^{t_1} \| u(t)\|^2_V \,dt \\
 & \quad
 + (C_1 +2) \int_0^{t_1} \|u'(t)\|^2_H \,dt
 + \int_0^{t_1} \|h(t)\|^2_H \,dt.
 \end{align*}
 Further, it follows from (\ref{eq:coercivity}) that
 $$
 LHS = \|u'(t_1)\|^2_H  + a_0(t_1, u(t_1), u(t_1))
 \geq \|u'(t_1)\|^2_H  + \mu \|u(t_1)\|^2_V - \lambda \|u(t_1)\|^2_H,
 $$
 and therefore (\ref{eq:srednja}) yields
 \begin{align*}
 \|u'(t_1)\|^2_H + \mu \|u(t_1)\|^2_V
 & \leq
 \lambda \|u(t_1)\|^2_H + C_0 \|f_1\|^2_V +\|f_2\|^2_H
 + \int_0^{t_1} \|h(t)\|^2_H \,dt \\
 & \qquad
 + (C_0' +C_1+ C_L) \int_0^{t_1} \| u(t)\|^2_V \,dt
 + (C_1 +2) \int_0^{t_1} \|u'(t)\|^2_H \,dt.
 \end{align*}
 As shown in \cite{HoermannOparnica09} we have that
 $\|u(t)\|^2_H \leq (1+t)(\|f_1\|^2_V + \int_0^{t} \|u'(s)\|^2_H \,ds)$,
 hence
 \begin{align*}
 \|u(t_1)\|^2_V + \|u'(t_1)\|^2_H \leq D_T \|f_1\|^2_V +
 \frac{1}{\nu} \left(
 \|f_2\|^2_H + \int_0^{t_1} \|h(t)\|^2_H \,dt
 \right)
 + F_T \int_0^{t_1} (\| u(t)\|^2_V + \|u'(t)\|^2_H )\,dt.
 \end{align*}
 where $\nu:=\min\{1,\mu\}$,
 $D_T:=\frac{C_0 + \lambda (1+ T)}{\nu}$ and
 $F_T := \max \{\frac{C_0' +C_1+ C_L}{\nu},
 \frac{C_1 +2+ \lambda (1+ T)}{\nu} \}$.
 The claim now follows from Gronwall's lemma.
 \end{proof}

 As a consequence of Proposition \ref{prop:EnergyEstimates},
 one also has uniqueness of the solution in Lemma \ref{lemma:m-a}.

 \begin{theorem}
 Under the assumptions of Lemma \ref{lemma:m-a} there exists a
 unique $u\in E_V$ satisfying the regularity conditions
 $u'\in E_V$ and $u''\in L^2(0,T;V')$,
 and solving the abstract initial value problem
 (\ref{eq:avp-L})-(\ref{eq:avp-L-ic}).
 Moreover, $u\in\cC([0,T];V)$ and $u'\in\cC([0,T];H)$.
 \end{theorem}

 \begin{proof}
 Since existence of a solution is proved in Lemma \ref{lemma:m-a},
 it remains to show uniqueness part of the theorem. Thus, let
 $u$ and $w$ be solutions to the abstract initial value problem
 (\ref{eq:avp-L})-(\ref{eq:avp-L-ic}), satisfying the regularity
 conditions $u',w'\in E_V$ and $u'',w''\in L^2(0,T;V')$. Then
 $u-w$ is a solution to the homogeneous abstract problem with
 vanishing initial data
 \begin{align*}
 &\dis{(u-w''(t)}{v} + a(t,(u-w)(t),v) + \dis{L(u-w)(t)}{v}= 0,
 \qquad \forall\,  v\in V,\,
 \mbox{ for a.e. } t \in (0,T), \\
 & (u-w)(0)=0, \qquad (u-w)'(0) = 0.
 \end{align*}
 Moreover, according to Proposition \ref{prop:EnergyEstimates},
 $u-w$ satisfies the energy estimates (\ref{eq:EE}) with
 $f_1=f_2=h\equiv 0$. This implies uniqueness of the solution.
 \end{proof}

 \subsection{Basic properties of the operator $L$}
 \label{ssec:L}

 In this subsection we analyze our particular form of the operator
 $L$, relevant to the problem described in the Introduction.
 Therefore, we consider an operator of convolution type
 and seek for conditions which guarantee estimate
 (\ref{eq:L-estimate}).

 \begin{lemma} \label{lemma:L_1}
 Let $l\in L^2_{loc}(\R)$ with $\supp l \subset [0,\infty)$. Then
 for all $T_1\in[0,T]$, the operator $L$ defined by $Lu(x,t) :=
 \int_0^t l(s) u(x,t-s)\,ds$ maps $L^2(0,T_1;H)$ into itself, and
 (\ref{eq:L-estimate}) holds with $C_L = \|l\|_{L^2(0,T)}\cdot T $.
 \end{lemma}

 \begin{remark}
 We may think of $u$ being extended by $0$ outside $[0,T]$ to a
 function in $L^2(\R;H)$, and then identify $Lu$ with $l\ast_t u$.
 \end{remark}

 \begin{proof}
 Integration of $\|Lu(t)\|_H^2 \leq \int_0^t |l(t-s)|\|u(s)\|_H \,ds$
 from $0$ to $T_1$, $0<T_1\leq T$, yields
 \begin{multline*}
 \left(
 \int_0^{T_1}\|Lu(t)\|_H^2 \,dt
 \right)^{1/2}
  \leq \left(\int_0^{T_1} (\int_0^t |l(t-s)|\|u(s)\|_H \,ds)^2 \,dt\right)^{1/2} \\
  \leq \left(\int_0^{T_1} (\int_0^{T_1} |l(t-s)|\|u(s)\|_H \,ds)^2 \,dt\right)^{1/2}
  \leq \int_0^{T_1} \left(\int_0^{T_1} |l(t-s)|^2 \|u(s)\|^2_H \,dt\right)^{1/2} \,ds \\
  = \int_0^{T_1} \left( \int_0^{T_1} |l(t-s)|^2 \,dt\right)^{1/2} \|u(s)\|_H \,ds
  = \|l\|_{L^2(0,T_1)}\cdot  \|u\|_{L^1(0,T_1; H)} \\
 \leq \|l\|_{L^2(0,T)}\cdot T \cdot \|u\|_{L^2(0,T_1; H)},
 \end{multline*}
 where we have used the support property of $l$, Minkowski's
 inequality for integrals (c.f \cite[p.\ 194]{Folland}), and
 the Cauchy-Schwartz inequality.
 \end{proof}

 In the following lemma we discuss a regularization of $L$, which
 will be used in Section \ref{ssec:Colombeau sol}.

 \begin{lemma} \label{lemma:reg L}
 Let $l\in L^1_{loc}(\R)$ with $\supp l \subset [0,\infty)$.
 Let $\rho\in\cD(\R)$ be a mollifier ($\supp \rho\subset B_1(0)$,
 $\int \rho =1$). Define $\rho_\eps(t):=\gamma_\eps
 \rho(\gamma_\eps t)$, with $\gamma_\eps>0$ and
 $\gamma_\eps\to\infty$ as $\eps\to 0$,
 $l_\eps := l*\rho_\eps $ and
 $\tilde{L}_\eps u(t):= (l_\eps *_t u)(t)$, for
 $u\in E_H$. Then
 $\forall\,p\in[1,\infty)$,
 $l_\eps\in L^p_{loc}(\R)$ and $l_\eps\to l$
 in $L^1_{loc}(\R)$.
 \end{lemma}

 \begin{proof}
 Let $K$ be a compact subset of $\R$. Then
 \begin{multline*}
 \|l_\eps\|_{L^p(K)}
  = \|l*_t\rho_\eps\|_{L^p(K)}
  = \left(\int_K |\int_{-\infty}^\infty
 l(\tau)\rho_\eps(t-\tau)\,d\tau|^p \,dt\right)^{1/p} \\
  \leq \left(\int_K \left(\int_{-\infty}^\infty
 |l(\tau)||\rho_\eps(t-\tau)|\,d\tau\right)^p \,dt\right)^{1/p}
  \leq \left(\int_K \left( \int_{K+ B_1(0)}
 |l(\tau)||\rho_\eps(t-\tau)|\,d\tau \right)^p \,dt\right)^{1/p} \\
  \leq \int_{K+ B_1(0)} \left(\int_K
 |l(\tau)|^p|\rho_\eps(t-\tau)|^p \,dt\right)^{1/p} \,d\tau
  = \int_{K+ B_1(0)} |l(\tau)|\left(\int_K
 |\rho_\eps(t-\tau)|^p \,dt\right)^{1/p} \,d\tau \\
  = \int_{K+ B_1(0)} |l(\tau)| \|\rho_\eps\|_{L^p(B_1(0))} \,d\tau
  = \|l\|_{L^1(K+B_1(0))} \|\rho_\eps\|_{L^p(B_1(0))} \\
  = \|l\|_{L^1(K+B_1(0))}
 \cdot \gamma_\eps^{1-\frac{1}{p}} \cdot
 \|\rho\|_{L^p(B_1(0))}.
 \end{multline*}
 where the second inequality follows from the support properties
 of $l$ and $\rho$ ($t-\tau\in B_1(0), t\in K$ implies
 $\tau\in K+B_1(0)$), while for the third inequality we used
 Minkowski's inequality for integrals.
 Further, we shall show that $l_\eps\to l$ in $L^1_{loc}(\R)$.
 Let $K\subset\subset\R$. We claim that
 $\int_K|l_\eps-l|\to 0 $, as $\eps\to 0$. Indeed,
\begin{multline*}
 \int_K  |\int_\R l(t-s)\rho_\eps(s)\,ds -
  l(t)\cdot\int_\R\rho_\eps(s)\,ds|\,dt
  =  \int_K |\int_\R (l(t-s)-l(t))\rho_\eps(s)\,ds|\,dt\\
  \stackrel{[\gamma_\eps s=\tau]}{=}
 \int_K |\int_\R (l(t-\frac{\tau}{\gamma_\eps}) -
  l(t))\rho(\tau)\,d\tau|\,dt
 \leq
 \int_K \int_\R |l(t-\frac{\tau}{\gamma_\eps}) -
   l(t)||\rho(\tau)|\,d\tau \,dt\\
  =
 \int_\R |\rho(\tau)| \int_K  |l(t-\frac{\tau}{\gamma_\eps})-l(t)|\,dt \,d\tau.
\end{multline*}
 By \cite[Prop.\ 8.5]{Folland}, we have that
 $\|l(\cdot-\frac{\tau}{\gamma_\eps})-l\|_{L^1(K)}\to 0$, as
 $\eps\to 0$ and therefore the integrand converges to 0 pointwise
 almost everywhere. Since it is also bounded
 by $2|\rho(\tau)|\|l\|_{L^1(K)}\in L^1(\R)$,
 Lebesgue's dominated convergence theorem implies the result.
 \end{proof}

 \section{Weak and generalized solutions of the model equations}
 \label{sec:EBmodel}

 We now come back to the problem
 (\ref{eq:PDE})-(\ref{eq:FDE})-(IC)-(BC) or
 (\ref{eq:IntegroPDE})-(IC)-(BC), and hence
 need to provide assumptions which guarantee
 that it can be interpreted in the form (\ref{eq:avp-L}), in order
 to the apply results obtained above.
 For that purpose we need to prescribe the regularity
 of the functions $c$ and $b$ which appear in $Q$.
 In Section \ref{ssec:Colombeau sol} we shall use these
 results on the level of representatives to prove existence of
 solutions in the Colombeau generalized setting.

 Thus, let
 $H := L^2(0,1)$ with the standard scalar product
 $\dis{u}{v}=\int_0^1u(x)v(x)\,dx$ and $L^2$-norm
 denoted by $\|\cdot\|_H$. Let $V$ be the Sobolev space
 $H^2_0((0,1))$, which is the completion of the space of compactly
 supported smooth functions $C^\infty_c((0,1))$
 with respect to the norm
 $\|u\|_{2} = (\sum_{k= 0}^2 \|u^{(k)}\|^2)^{1/2}$ (and inner
 product $(u,v) \mapsto \sum_{k=0}^2\dis{u^{(k)}}{v^{(k)}}$).
 Then  $V'= H^{-2}((0,1))$, which consists of distributional
 derivatives up to second order of functions in $L^2(0,1)$,
 and $V\hookrightarrow H \hookrightarrow V'$ forms a Gelfand
 triple. With this choice of spaces $H$ and $V$ we also have that
 $E_V=L^2(0,T; H^2_0((0,1)))$ and $E_H=L^2((0,1)\times (0,T))$.

 Let
 \begin{equation} \label{HypothesisOn_c_and_b}
 c\in L^\infty(0,1) \mbox{ and real},
 \qquad b\in C([0,T];L^{\infty}(0,1)),
 \end{equation}
 and suppose that there exist constants $c_1 > c_0 > 0$ such that
 \begin{equation} \label{addHypothesisOn_c}
 0 < c_0\leq c(x)\leq c_1, \qquad \mbox{ for almost every } x.
 \end{equation}
 For $t\in [0,T]$ we define the bilinear forms
 $a(t,\cdot,\cdot)$, $a_0(t,\cdot,\cdot)$
 and $a_1(t,\cdot,\cdot)$ on $V\times V$ by
 \begin{equation} \label{sesforma}
 a_0(t,u,v) = \dis{c(x)\, \pd_x^2u}{\pd_x^2v}, \qquad
 a_1(t,u,v) = \dis{b(x,t)\, \pd_x^2u}{v},
 \qquad u,v\in V,
 \end{equation}
 and
 \beq \label{sesform2}
 a(t,u,v) = a_0(t,u,v) + a_1(t,u,v).
 \eeq
 Properties (\ref{HypothesisOn_c_and_b}),
 (\ref{addHypothesisOn_c}) imply that $a_0$, $a_1$ defined as in (\ref{sesforma})
 satisfy the conditions of Assumption \ref{Ass1}
 (cf.\ \cite[proof of Th.\ 2.2]{HoermannOparnica09}).
 The specific form of the operator $L$ is designed to achieve
 equivalence of the system (\ref{eq:PDE})-(\ref{eq:FDE}) with the
 equation (\ref{eq:IntegroPDE}), which we show in the sequel.

 Let $\cS'_+$ denote the space of Schwartz' distributions
 supported in $[0,\infty)$. It is known (c.f. \cite{Oparnica02})
 that for given $z\in \cS'_+$ there is a unique $y\in\cS'_+$
 such that $D_t^\alpha z + z = \theta D_t^\alpha y + y$. Moreover,
 it is given by $y=\tilde{L}z$, where $\tilde{L}$ is linear
 convolution operator acting on $\cS'_+$ as
 \begin{equation} \label{operatorLonEs}
 \tilde{L}z: =\ILT \left(\frac{1 + s^{\alpha}}{1 + \theta s^{\alpha}}\right)
 \ast_t z, \qquad z\in\cS'_+.
 \end{equation}
 The following lemma extends the operator $\tilde{L}$ to the space
 $E_H$.

 \begin{lemma}
 Let $\tilde{L}:\cS'_+ \to \cS'_+$ be defined as in (\ref{operatorLonEs}).
 Then $\tilde{L}$ induces a continuous operator $L=\Id+L_\alpha$
 on $E_H$, where $L_\alpha$ corresponds to convolution in time
 variable with a function $l_\alpha\in L^1_{loc}([0,\infty))$.
 \end{lemma}

 \begin{proof}
 Recall that for the Mittag-Leffler function $e_\alpha(t,\lambda)$,
 defined by
 $$
 e_\alpha(t,\lambda) =
 \sum_{k=0}^\infty \frac{(-\lambda t^\alpha)^k}{\Gamma(\alpha k+1)},
 $$
 we have that
 $\LT (e_{\alpha}(t,\lambda))(s) =\frac{s^{\alpha-1}}{s^{\alpha} + \lambda}$,
 $e_{\alpha}\in C^\infty((0,\infty))\cap C([0,\infty))$ and
 $e'_{\alpha} \in C^\infty((0,\infty))\cap L^1_{loc}([0,\infty))$
 (cf.\ \cite{MainardiGorenflo2000}).
 Also,
 $$
 \ILT \left(\frac{1+s^{\alpha}}{1+\theta s^{\alpha}}\right)(t)=
 \ILT\left(
 1 + \frac{(1-\theta)s^\alpha}{\theta (s^\alpha + \frac{1}{\theta})}
 \right)(t)
 = \delta(t) + \left(\frac{1}{\theta}-1\right)
 e_\alpha'\left(t,\frac{1}{\theta}\right)
 =: \delta(t)+l_\alpha(t).
 $$

 Let $ u\in E_H$. Then
 \begin{equation}
 \ILT \left(\frac{1+ s^{\alpha}}{1+ \theta s^{\alpha}}\right)(\cdot)
 \ast_t u (x,\cdot)
 = u(x,\cdot) + \left(\frac{1}{\theta}-1\right) e_\alpha'
 \ast u(x,\cdot)
 \end{equation}
 is an element in $L^2(0,T)$ for almost all $x$
 (use Fubini's theorem, $e'_{\alpha} \in L^1(0,T)$ and
 $L^1\ast L^p\subset L^p$ (cf.\ \cite{Folland})).
 Extend this to a measurable function on $(0,1)\times(0,T)$,
 denoted by $Lu$.
 By Young's inequality we have
 \begin{align*}
 \|(Lu)(x,\cdot)\|_{L^2(0,T)}
 & \leq \|u(x,\cdot)\|_{L^2(0,T)}
 +|\frac{1}{\theta}-1|\|e_\alpha'\ast u(x,\cdot)\|_{L^2(0,T)} \\
 & \leq \|u(x,\cdot)\|_{L^2(0,T)} + |\frac{1}{\theta}-1|
 \|e'_\alpha\|_{L^1(0,T)} \|u(x,\cdot)\|_{L^2(0,T)},
 \end{align*}
 hence,
 \begin{equation} \label{Lbound}
 \|Lu\|_{E_H} \leq (1 + |\frac{1}{\theta}-1|
 \|e'_\alpha\|_{L^1(0,T)}) \|u\|_{E_H}.
 \end{equation}
 Thus, $Lu\in E_H$ and $L$ is continuous on $E_H$.
 \end{proof}

 We may write
 \beq \label{eq:opL}
 Lu := (\Id+L_\alpha)u
 = l\ast_t u
 = (\delta + l_\alpha)\ast_t u
 \quad \mbox{ with } \quad
 L_\alpha u := l_\alpha\ast_t u,
 \;
 l_\alpha:= (\frac{1}{\theta} - 1)
 e_\alpha' (t,\frac{1}{\theta}),
 \eeq
 and therefore the model system (\ref{eq:PDE})-(\ref{eq:FDE}) is
 equivalent to Equation (\ref{eq:IntegroPDE}).

 \subsection{Weak solutions for $L^{\infty}$ coefficients}
 \label{ssec:weak sol}

 Now we are in a position to apply the abstract results from the
 previous section to the original problem.

 \begin{theorem} \label{th:weak sol}
 Let $b$ and $c$ be as in (\ref{HypothesisOn_c_and_b}) and
 (\ref{addHypothesisOn_c}). Let the bilinear form
 $a(t,\cdot,\cdot)$, $t\in [0,T]$, be defined by
 (\ref{sesforma}) and (\ref{sesform2}), and the operator $L$ as in
 (\ref{eq:opL}). Let $f_1\in H_0^2((0,1))$, $f_2\in L^2(0,1)$ and $h\in L^2((0,1)\times(0,T))$. 
 Then there exists a unique $u\in L^2(0,T;H_0^2(0,1))$ satisfying
 \beq \label{sol_reg}
 u' = \frac{du}{dt} \in L^2(0,T;H_0^2(0,1)),
 \qquad
 u'' = \frac{d^2 u}{dt^2}\in L^2(0,T;H^{-2}(0,1)),
 \eeq
 and solving the initial value problem
 \begin{align}
 & \dis{u''(t)}{v} + a(t,u(t),v) + \dis{Lu(t)}{v} = 0,
 \qquad \forall\, v\in H_0^2((0,1)),\, t \in (0,T), \label{vf}\\
 & u(0)=f_1, \qquad u'(0) = f_2. \label{vf_ini}
 \end{align}
 (Note that, as in the abstract version, since (\ref{sol_reg})
 implies $u \in C([0,T],H^2_0((0,1)))$ and  $u' \in C([0,T],H^{-2}((0,1)))$ it makes sense to
 evaluate $u(0)\in H^2_0((0,1))$ and
 $u'(0) \in H^{-2}((0,1))$ and (\ref{vf_ini}) claims that
 these equal $f_1$ and $f_2$, respectively.)
 \end{theorem}

 \begin{proof}
 We may apply Lemma \ref{lemma:m-a} because the bilinear form $a$
 and the operator $L$ satisfy Assumption \ref{Ass1} and condition
 (\ref{eq:L-estimate}).
 The latter is true according to (\ref{Lbound}) with
 $C_L=(1 + |\frac{1}{\theta}-1| \|e'_\alpha\|_{L^1(0,T)})
 =1+\|l_\alpha\|_{L^2(0,T)}$.
 As noted earlier, the bilinear forms $a$, $a_0$ and $a_1$ are as
 in \cite[(20) and (21)]{HoermannOparnica09}.
 Moreover, it follows as in the proof of \cite[Theorem\ 2.2]{HoermannOparnica09} that
 $a$ satisfies Assumption \ref{Ass1} with
 \beq \label{eq:constants}
 C_0:=\|c\|_{L^\infty(0,1)},
 \quad
 C_0':=0,
 \quad
 C_1:=\|b\|_{L^\infty((0,1)\times(0,T))},
 \quad
 \mu:=\frac{c_0}{2},
 \quad
 \lambda:=C_{1/2}\cdot c_0,
 \eeq
 where $C_{1/2}$ is corresoponding constant form Ehrling's lemma. 
\end{proof}

 We briefly recall two facts about the solution $u$ obtained in
 Theorem \ref{th:weak sol} (as noted similarly in
 \cite[Section 2]{HoermannOparnica09}):

 (i) Since $u(.,t) \in H^2_0((0,1))$ for all
 $t \in [0,T]$ and $H^2_0((0,1))$ is continuously embedded in
 $\{v \in C^1([0,1]): v(0,t) = v(1,t) = 0,
 \pd_x v(0,t) = \pd_x v(1,t) = 0\}$  (\cite[Corollary 6.2]{Wloka})
 the solution $u$ automatically satisfies the boundary conditions.

 (ii) The properties in  (\ref{sol_reg}) imply that $u$ belongs to
 $C^1([0,T],H^{-2}((0,1))) \cap L^2((0,T)\times(0,1))$, which is
 a subspace of  $\cD'((0,1)\times(0,T))$. Thus in case of smooth
 coefficients $b$ and $c$ we obtain a distributional solution to the
 ``integro-differential'' equation
 $$
 \pd_t^2 u + \pd^2_x(c \,\pd^2_x)u + b \,\pd_x^2 u + l *_t u = h.
 $$

 \subsection{Colombeau generalized solutions}
 \label{ssec:Colombeau sol}

 We will prove unique solvability  of Equation (\ref{eq:IntegroPDE})
 (or equivalently, of Equations (\ref{eq:PDE})-(\ref{eq:FDE}))
 with (IC) and (BC) for $u \in \cG_{H^{\infty}(X_T)}$ when
 $b,c, f_1, f_2, g$ and $h$ are Colombeau generalized functions, where $X_T:=(0,1)\times (0,T)$.

 In more detail, we find a unique solution
 $u\in\cG_{H^{\infty}(X_T)}$ to the equation
 $$
 \pd^2_tu + Q(t,x,\pd_x)u + Lu = h,
 \qquad \mbox{ on } X_T
 $$
 with initial conditions
 $$
 u|_{t=0} = f_1\in \cG_{H^{\infty}((0,1))},
 \qquad \pd_t u|_{t=0} = f_2\in \cG_{H^{\infty}((0,1))}
 $$
 and boundary conditions
 $$
 u|_{x=0} =u|_{x=1}=0,
 \qquad \pd_x u|_{x=0} = \pd_x u|_{x=1}=0.
 $$
 Here $Q$ is a partial differential operator on $\cG_{H^{\infty}(X_T)}$ with generalized
 functions as coefficients, defined by
 its action on representatives in the form
 $$
 (u_\eps)_\eps \mapsto \left(
 \pd_x^2(c_\eps(x)\pd_x^2 u_\eps))
 + b_\eps(x,t)\pd_x^2 (u_\eps)
 \right)_\eps =: (Q_\eps u_\eps)_\eps.
 $$
 Furthermore, the operator $L$ corresponds to convolution on
 the level of representatives with regularizations of $l$
 as given in Lemma \ref{lemma:reg L}:
 $$
 (u_\eps)_\eps \mapsto \left(
 l_\eps\ast_t u_\eps (t)
 \right)_\eps =: (L_\eps u_\eps)_\eps,
 $$
 where $l_\eps=l\ast \rho_\eps$, with $\rho_\eps$ introduced in
 Lemma \ref{lemma:reg L}.

 \begin{lemma}
 (i)\ If $l\in L^2_{loc}(\R)$ and $l$ is $\cC^\infty$ in
 $(0,\infty)$ then $L$ is a continuous operator on $H^{\infty}(X_T)$.
 Thus $(u_\eps)_\eps\mapsto (Lu_\eps)_\eps$ defines a linear map
 on $\cG_{H^\infty(X_T)}$.

 (ii)\ If $l\in L^1_{loc}(\R)$ then $\forall\,\eps\in(0,1]$ the
 operator $L_\eps$ is continuous on $H^{\infty}(X_T)$ and
 $(u_\eps)_\eps\mapsto (Lu_\eps)_\eps$ defines a linear map
 on $\cG_{H^\infty(X_T)}$.
 \end{lemma}

 \begin{proof}
 (i)\ From Lemma \ref{lemma:L_1} with $H=L^2(0,1)$ we have that
 $L$ is continuous on $L^2(X_T)$ with operator norm $\|L\|_{op}\leq
 T\cdot \|l\|_{L^2(0,T)}$.

 Let $u\in H^\infty(X_T)$ and $Lu(x,t)=\int_0^t l(s)u(x,t-s)\,ds$.
 We have to show that all derivatives of $Lu$ with respect to both
 $x$ and $t$ are in $L^2(X_T)$.
 \begin{itemize}
 \item $\pd_x^l Lu(x,t)=\int_0^t l(s) \pd_x^l u(x,t-s)\,ds$, and
 hence, $\|\pd_x^l Lu\|_{L^2(X_T)}\leq T\cdot \|l\|_{L^2(0,T)}
 \|\pd_x^l u\|_{L^2(X_T)}$.
 \item $\pd_t^k \pd_x^l Lu=\pd_t^k L(\pd_x^l u)$, and since the
 estimates for $L(\pd_x^l u)$ are known it suffices to consider
 only terms $\pd_t^k Lu$. For the first order derivative we have
 $$
 \pd_t Lu(x,t) = l(t)u(x,0) + \int_0^t l(s)\pd_t u(x,t-s)\, ds\\
 $$
 and therefore
 \beas
 \|\pd_t Lu\|_{L^2(X_T)} &\leq&
 \|l\|_{L^2(0,T)} \|u(\cdot,0)\|_{L^2(0,1)}
 + T\cdot \|l\|_{L^2(0,T)} \|\pd_t u\|_{L^2(X_T)} \\
 &\leq& \|l\|_{L^2(0,T)} (\|u\|_{H^m(X_T)}
 + T\cdot \|u\|_{H^1(X_T)}),
 \eeas
 where we have used the fact that $\Tr:H^\infty(X_T)\to
 H^\infty((0,1))$, $u\mapsto u(\cdot,0)$ is continuous, and more
 precisely, $\Tr:H^m(X_T)\to H^{m-1}((0,1))$ with estimates
 $\|\pd_x^l \pd_t^k u(\cdot,0)\|_{L^2(0,1)}\leq \|u\|_{H^m(X_T)}$,
 $m=m(k,l)$.

 Higher order derivatives involve terms $l^{(r)}(t) \pd_t^p u(x,0),
 \ldots, \int_0^t l(s) \pd_t^p u(x,t-s)\, ds$, which can be
 estimated as above.
 \end{itemize}

 (ii)\  From Lemma \ref{lemma:reg L} it follows that $l_\eps\in
 L^2_{loc}(\R)$, and $\|l_\eps\|_{L^2(0,T)}\leq\gamma_\eps^\frac12
 \cdot \|l\|_{L^1(0,T)} \|\rho\|_{L^2(0,T)}$. From Lemma
 \ref{lemma:L_1} we know that $L_\eps$ is continuous $X_T\to X_T$,
 with $\|L_\eps\|_{op}\leq T\cdot \|l_\eps\|_{L^2(0,T)} \leq
 T\cdot \gamma_\eps^\frac12 \cdot \|l\|_{L^1(0,T)}
 \|\rho\|_{L^2(0,T)}$, which is moderate. We can now proceed as in
 (i) to produce estimates of $\|L_\eps u_\eps\|_{H^r(X_T)}$,
 $\forall\,r\in\N$, always replacing $\|l\|_{L^2(0,T)}$ by
 $\gamma_\eps^\frac12 \cdot \|l\|_{L^1(0,T)} \|\rho\|_{L^2(0,T)}$
 factors. Since $\gamma_\eps\leq \eps^{-N}$ it follows that
 $(L_\eps u_\eps)_\eps \in \cE_{H^\infty(X_T)}$.
 \end{proof}

 \begin{remark} The function $l$ as defined in (\ref{eq:opL})
 belongs to $L^2_{loc}(\R)$, if $\alpha>1/2$, and to
 $L^1_{loc}(\R)$, if $\alpha\leq 1/2$  (which follows from
 the explicit form of $e_\alpha'(t,\frac{1}{\theta})$).
 This means that in case $\alpha > 1/2$ we could in fact
 define the operator $L$ without regularization of $l$.
 \end{remark}

 As in the classical case we also have to impose a
 condition to ensure compatibility of initial with boundary
 values, namely (as equation in generalized numbers)
 \beq \label{compatibility}
 f_1(0) = f_1(1) = 0.
 \eeq
 Note that if $f_1 \in \cG_{H^\infty((0,1))}$ satisfies
 (\ref{compatibility}) then there is some representative
 $(f_{1,\eps})_\eps$ of $f_1$ with the property $f_{1,\eps} \in
 H^2_0((0,1))$ for all $\eps \in\, (0,1)$ (cf.\ the
 discussion right below Equation (28) in \cite{HoermannOparnica09}).

 Motivated by condition (\ref{addHypothesisOn_c}) above
 on the bending stiffness we assume the following about $c$:
 There exist real constants $c_1 > c_0 > 0$ such that
 $c\in \cG_{H^{\infty}(0,1))}$
 possesses a representative $(c_{\eps})_\eps$ satisfying
 \beq \label{eq:c-1-2}
 0 < c_0 \leq c_{\eps}(x) \leq c_1
 \qquad \forall\, x\in (0,1), \forall\, \eps \in\, (0,1].
 \eeq
 (Hence any other representative of $c$ has upper and
 lower bounds of the same type.)

 As in many evolution-type problems with Colombeau generalized
 functions we also need the standard assumption that $b$ is of
 $L^\infty$-log-type (similar to \cite{MO-89}), which means that
 for some (hence any) representative $(b_\eps)_\eps$ of
 $b$ there exist $N\in\N$ and $\eps_0 \in (0,1]$ such that
 \beq \label{log_type}
 \|b_\eps\|_{L^\infty(X_T)}
 \leq N\cdot \log(\frac{1}{\eps}),
 \qquad 0 < \eps \leq \eps_0.
 \eeq
 It has been noted already in \cite[Proposition 1.5]{MO-89}
 that log-type regularizations of distributions are obtained
 in a straight-forward way by convolution with logarithmically
 scaled mollifiers.

 \begin{theorem} \label{th:main}
 Let $b\in \cG_{H^{\infty}(X_T)}$ be of $L^\infty$-log-type
 and $c\in \cG_{H^{\infty}(0,1))}$ satisfy (\ref{eq:c-1-2}).
 Let $\gamma_\eps=O(\log \frac{1}{\eps})$.
 For any $f_{1} \in \cG_{H^{\infty}((0,1))}$ satisfying
 (\ref{compatibility}), $f_{2}\in \cG_{H^{\infty}((0,1))}$,
 $h\in \cG_{H^{\infty}(X_T)}$ and
 $l\in\cG_{H^{\infty}((0,1))}$,
 there is a unique solution $u \in \cG_{H^{\infty}(X_T)}$ to
 the initial-boundary value problem
 \begin{align*}
 & \pd^2_tu + Q(t,x,\pd_x)u + Lu = h, \\
 & u|_{t=0} = f_1, \quad \pd_t u|_{t=0} = f_2, \\
 & u|_{x=0} =u|_{x=1}=0, \quad
 \pd_x u|_{x=0} = \pd_x u|_{x=1}=0.
 \end{align*}
 \end{theorem}

 \begin{proof} Thanks to the preparations a considerable part of
 the proof may be adapted from the corresponding proof in
 \cite[Theorem 3.1]{HoermannOparnica09}.
 Therefore we give details only for the first part and sketch
 the procedure from there on.

 {\bf Existence:} \enspace \enspace
 We choose representatives $(b_\eps)_\eps$ of $b$ and
 $(c_{\eps})_\eps$ of $c$ satisfying (\ref{eq:c-1-2})
 and (\ref{log_type}). Further let
 $(f_{1\eps})_\eps$, $(f_{2\eps})_\eps$, $(l_{\eps})_\eps$,
 and $(h_{\eps})_\eps$ be
 representatives of $f_1$,$f_2$, $l$, and $g$, respectively,
 where  we may assume $f_{1,\eps} \in H^2_0((0,1))$ for all
 $\eps \in \,(0,1)$ (cf.\ (\ref{compatibility})).

 For every $\eps\in (0,1]$  Theorem
 \ref{th:weak sol}
 provides us with a unique
 solution $u_{\eps}\in H^1((0,T),H^2_0((0,1))) \cap
 H^2((0,T),H^{-2}((0,1)))$ to
 \begin{align}
 P_{\eps}u_{\eps} &: = \pd^2_t u_{\eps} +
 Q_{\eps}(t,x,\pd_x)u_{\eps} + L_\eps u_\eps
 = h_{\eps}
 \qquad \mbox{ on } X_T, \label{Peps}\\
 & u_{\eps}|_{t=0} = f_{1\eps},
 \qquad \pd_tu_{\eps}|_{t=0} = f_{2\eps}. \nonumber
 \end{align}
 In particular, we have $u_{\eps}\in C^1([0,T],H^{-2}((0,1))) \cap
 C([0,T],H^2_0((0,1)))$.

 Proposition \ref{prop:EnergyEstimates} implies the energy estimate
 \begin{equation}\label{eq:EEeps1}
 \|u_{\eps}(t)\|_{H^2}^2 + \|u_{\eps}'(t)\|_{L^2}^2 \leq
 \big( D_T^\eps\, \|f_{1\eps}\|_{H^2}^2 + \|f_{2,\eps}\|_{L^2}^2
 + \int_0^t \|h_{\eps}(\tau)\|_{L^2}^2\,d\tau \big)
 \cdot \exp(t\, F_T^\eps),
 \end{equation}
 where with some $N$ we have as  $\eps \to 0$
 \begin{align}
 D_T^\eps &= (\|c_\eps\|_{L^{\infty}} +
 \lambda (1+T)) /\min(\mu,1)
 = O(\|c_\eps\|_{L^{\infty}}) = O(1)
 \label{const_D_eps} \\
 F_T^\eps &= \frac{\max \{\|b_\eps\|_{L^{\infty}}+C_{L,\eps},
 \|b_\eps\|_{L^{\infty}} +2+\lambda(1+T) \}}{\min(\mu,1)}
 = O(C_{L,\eps}+\|b_\eps\|_{L^{\infty}})
 = O(\log(\eps^{-N})), \label{const_F_eps}
 \end{align}
 since $\mu$ and $\lambda$ are independent of
 $\eps$, and $C_{L,\eps} = O(\log \frac{1}{\eps})$
 (cf.\ (\ref{eq:constants})).

 By moderateness of the initial data
 $f_{1\eps}$, $f_{2\eps}$ and of the right-hand side $h_{\eps}$
 the inequality (\ref{eq:EEeps1}) thus implies that there
 exists $M$ such that for  small $\eps > 0$ we have
 \begin{equation}\label{EEeps}
 \|u_{\eps}\|^2_{L^2(X_T)} +
 \|\pd_x u_{\eps}\|^2_{L^2(X_T)}
 + \| \pd_x^2 u_{\eps}\|^2_{L^2(X_T)}
 + \|\pd_t u_{\eps}\|^2_{L^2(X_T)} =
 O (\eps^{-M}),
 \qquad \eps\to 0.
 \end{equation}

 From here on the remaining chain of arguments proceeds along the lines of the
 proof in \cite[Theorem 3.1]{HoermannOparnica09}.
 We only indicate a few key points requiring certain adaptions.


The goal is to prove the following properties:
 \begin{itemize}
 \item[1.)]  For every $\eps \in (0,1]$ we have $u_{\eps}\in H^{\infty}(X_T)
 \subseteq C^{\infty}(\overline{X_T})$.
 \item[2.)] Moderateness, i.e.\ for all $l,k\in\N$ there is some
 $M\in\N$ such that for small $\eps > 0$
 \beq \tag{$T_{l,k}$} \label{T_lk}
 \|\pd^l_t\pd^k_xu_{\eps}\|_{L^2(X_T)} = O(\eps^{-M}).
 \eeq
 Note that (\ref{EEeps}) already yields (\ref{T_lk}) for
 $(l,k) \in \{ (0,0),(1,0),(0,1),(0,2)\}$.
 \end{itemize}

 \noindent\emph{Proof of 1.)}  Differentiating (\ref{Peps}) (considered
 as an equation in $\cD'((0,1)\times(0,T))$) with respect to $t$
 we obtain
 $$
 P_{\eps}(\pd_tu_{\eps}) = \pd_th_{\eps}
 - \pd_t b_\eps(x,t)\pd_x^2u_{\eps}
 - l_\eps(t) f_{1,\eps}
 =: \tilde{h}_\eps,
 $$
 where we used $\pd_t(L_\eps u_\eps)=
 L_\eps (\pd_t u_\eps) + l_\eps(t)u_\eps(0)$.
 We have $\tilde{h}_\eps \in H^1((0,T),L^2(0,1))$
 since
 $\pd_t h_{\eps}\in H^{\infty}(X_T)$,
 $l_\eps\in H^{\infty}((0,T))$,
 $f_{1,\eps}\in H^{\infty}((0,1))$,
 $\pd_t b_\eps(x,t)\in H^{\infty}(X_T) \subset W^{\infty,\infty}(X_T)$
 and  $\pd^2_xu_{\eps}\in H^1((0,T),L^2(0,1))$.
 Furthermore, since $Q_\eps$ depends smoothly on $t$ as a
 differential operator in $x$ and $u_\eps(0) = f_{1,\eps} \in
 H^{\infty}((0,1))$ we have
 \begin{align*}
 (\pd_t u_{\eps})(\cdot,0) & = f_{2,\eps} =:
 \tilde{f}_{1,\eps} \in H^{\infty}((0,1)),\\
 (\pd_t(\pd_tu_{\eps}))(\cdot,0) & =
 h_{\eps}(\cdot,0) - Q_{\eps}(u_{\eps}(\cdot,0))-L_\eps u_\eps (\cdot,0)=
 h_{\eps}(\cdot,0) - (Q_{\eps}+L_\eps) f_{1,\eps} :=
 \tilde{f}_{2,\eps} \in H^{\infty}((0,1)).
 \end{align*}
 Hence $\pd_tu_{\eps}$ satisfies an initial value problem for the
 partial differential operator $P_\eps$ as in (\ref{Peps}) with
 initial data
 $\tilde{f}_{1,\eps}$, $\tilde{f}_{2,\eps}$ and right-hand side
 $\tilde{h}_{\eps}$ instead. However, this time we have to use $V = H^2((0,1))$
 (replacing $H^2_0((0,1))$) and $H = L^2(0,1)$ in the abstract setting,
 which still can serve to define a Gelfand triple
 $V\hookrightarrow H \hookrightarrow V'$ (cf.\ \cite[Theorem 17.4(b)]{Wloka})
 and thus allows for application of Lemma \ref{lemma:m-a} and
 the energy estimate (\ref{eq:EE}) (with precisely the same constants).

 Therefore we obtain
 $\pd_tu_{\eps}\in H^1([0,T],H^2((0,1)))$,
 i.e.\ $u_{\eps}\in H^2((0,T),H^2((0,1)))$ and from the
 variants of (\ref{eq:EEeps1}) (with exactly the same constants
 $D_T^\eps$ and $F_T^\eps$) and (\ref{EEeps})
 with $\pd_t u_\eps$ in place of $u_\eps$ that for some $M$
 we have
 \begin{equation}
 \| \pd_t u_{\eps}\|^2_{L^2(X_T)} +
 \|\pd_x \pd_t u_{\eps}\|^2_{L^2(X_T)}
 + \| \pd_x^2 \pd_t u_{\eps}\|^2_{L^2(X_T)}
 + \|\pd_t^2 u_{\eps}\|^2_{L^2(X_T)} =
 O (\eps^{-M})\quad (\eps\to 0).
 \end{equation}
 Thus we have proved (\ref{T_lk}) with $(l,k) = (2,0), (1,1), (1,2)$ in
 addition to those obtained from (\ref{EEeps}) directly.

 The remaining part of the proof of property 1.) requires the
 exact same kind of adaptions in the corresponding parts in
 Step 1 of the proof of \cite[Th.\ 3.1]{HoermannOparnica09}
 and we skip its details here. In particular, along the way
 one also obtains that
 \begin{center}
 ($T_{l,k}$)  holds for
 derivatives of arbitrary $l$ and $k \leq 2$.
\end{center}

 \noindent\emph{Proof of 2.)} From the estimates achieved in proving 1.)
 and equation (\ref{Peps}) we deduce that
 $$
 k_\eps := \pd_x^2(c_\eps\, \pd_x^2 u_\eps) = h_\eps -
 b_\eps\, \pd_x^2 u_\eps - \pd_t^2 u_\eps - L_\eps u_\eps
 $$
 satisfies for all $l \in \N$ with some $N_l$ an estimate
 \begin{equation}\label{mod22}
 \| \pd_t^l  k_\eps \|_{L^2(X_T)}
 = O(\eps^{-N_l}) \qquad (\eps\to 0).
 \end{equation}

 Here we are again in the same situation as in Step 2 of
 the proof of \cite[Theorem 3.1]{HoermannOparnica09},
 where now $k_\eps$ plays the role of $h_\eps$ there.
 Skipping again details of completely analogous arguments
 we arrive at the conclusion that the class of $(u_\eps)_\eps$
 defines a solution to the initial value problem.

 Moreover, we have by construction that $u_\eps(t) \in H^2_0((0,1))$
 for all $t \in [0,T]$, hence $u(0,t) = u(1,t) = 0$ and
 $\pd_x u(0,t) = \pd_x u(1,t) = 0$ and thus $u$ also satisfies
 the boundary conditions.

 {\bf Uniqueness:}\enspace \enspace  If $u =
 [(u_\eps)_\eps]$ satisfies initial-boundary value problem with
 zero initial values and right-hand side, then we have for all
 $q\geq 0$
 $$
 \|f_{1,\eps}\| = O(\eps^q), \quad
 \|f_{2,\eps}\| = O(\eps^q), \quad
 \|h_{\eps}\|_{L^2(X_T)} = O(\eps^q) \qquad \text{as } \eps\to 0.
 $$
 Therefore the energy estimate (\ref{eq:EEeps1}) in combination with
 (\ref{const_D_eps})-(\ref{const_F_eps}) imply for all $q \geq 0$ an
 estimate
 $$
 \|u_\eps\|_{L^2(X_T)} = O(\eps^q) \quad (\eps \to 0),
 $$
 from which we conclude that $(u_\eps)_\eps \in  \cN_{H^{\infty}(X_T)}$, i.e., $u = 0$.
 \end{proof}



\begin{thebibliography}{10}

\bibitem{AF:03}
{Adams, R.~A., Fournier, J.}
\newblock {\em Sobolev Spaces}.
\newblock Elsevier, Oxford, second edition, 2003.

\bibitem{Atanackovic-book}
{Atanackovi\'c, T.\ M.}
\newblock {\em Stability Theory of Elastic Rods}.
\newblock World Scientific, Singapore, 1997.

\bibitem{BO:92}
{Biagioni, H. A., Oberguggenberger, M.}
\newblock Generalized solutions to the {K}orteweg-de {V}ries and the
  regularized long-wave equations.
\newblock {\em SIAM J. Math. Anal.}, {\bf 23}:923--940, 1992.

\bibitem{BiondiCaddemi}
{Biondi, B., Caddemi, S.}
\newblock {E}uler-{B}ernoulli beams with multiple singularities in the flexural
  stiffness.
\newblock {\em Eur. J. Mech., A, Solids}, {\bf 46}(5):789--809, 2007.

\bibitem{CCR}
{Colombeau, J.~F.}
\newblock Une multiplication g\'en\'erale des distributions.
\newblock {\em C. R. Acad. Sci. Paris, S\'er. I}, {\bf 296}:357--360, 1983.

\bibitem{c1}
{Colombeau, J.~F.}
\newblock {\em New Generalized Functions and Multiplication of Distributions}.
\newblock North Holland, Amsterdam, 1984.

\bibitem{DautrayLions-vol5}
{Dautray, R., Lions, J.L.}
\newblock {\em Mathematical Analysis and Numerical Methods for Science and
  Technology}, volume {\bf 5}.
\newblock Springer-Verlag, Berlin, 2000.

\bibitem{Folland}
{Folland, G. B.}
\newblock {\em Real Analysis: Modern Techniques and Their Applications}.
\newblock John Wiley \& Sons, Inc., New York, second edition, 1999.

\bibitem{garetto_duals}
{Garetto, C.}
\newblock Topological structures in {C}olombeau algebras: investigation of the
  duals of {${\mathcal G}_c(\Omega)$}, {${\mathcal G}(\Omega)$}, and
  {${\mathcal G}_{\mathcal S}(\Omega)$}.
\newblock {\em Monatsh. Math.}, {\bf 146}(3):203--226, 2005.

\bibitem{book}
{Grosser, M., Kunzinger, M., Oberguggenberger, M., Steinbauer, R.}
\newblock {\em Geometric Theory of Generalized Functions}, volume 537 of {\em
  Mathematics and its Applications}.
\newblock Kluwer Academic Publishers, Dordrecht, 2001.

\bibitem{GH:04}
{H\"{o}rmann, G.}
\newblock First order hyperbolic pseudodifferential equations with generalized
  symbols.
\newblock {\em J. Math. Anal. Appl.}, {\bf 293}:40--56, 2004.

\bibitem{HoermannOparnica07}
{H\"{o}rmann, G., Oparnica, Lj.}
\newblock Distributional solution concepts for the {E}uler-{B}ernoulli beam
  equation with discontinuous coefficients.
\newblock {\em Applic.\ Anal.}, {\bf 86}(11):1347--1363, 2007.

\bibitem{HoermannOparnica09}
{H\"{o}rmann, G., Oparnica, Lj.}
\newblock Generalized solutions for the {E}uler-{B}ernoulli model with
  distributional forces.
\newblock {\em J. Math. Anal. Appl.}, {\bf 357}(1):142--153, 2009.

\bibitem{LionsMagenes}
{Lions, J. L., Magenes, E.}
\newblock {\em Non-homogenious Boundary Value Problems and Applications},
  volume~{\bf 1, 2, 3}.
\newblock Springer-Verlag, Berlin, 1972.

\bibitem{MainardiGorenflo2000}
{Mainardi, F., Gorenflo, R.}
\newblock On {M}ittag-{L}effler-type functions in fractional evolution
  processes.
\newblock {\em J. Comput. Appl. Math.}, {\bf 118}(1-2):283--299, 2000.

\bibitem{MO-89}
{Oberguggenberger, M.}
\newblock Hyperbolic systems with discontinuous coefficients: generalized
  solutions and a transmission problem in acoustics.
\newblock {\em J.\ Math.\ Anal.\ Appl.}, {\bf 142}:452--467, 1989.

\bibitem{MOBook}
{Oberguggenberger, M.}
\newblock {\em Multiplication of Distributions and Applications to Partial
  Differential Equations}, volume~{\bf 259} of {\em Pitman Research Notes in
  Mathematics}.
\newblock Longman, Harlow, U.K., 1992.

\bibitem{Oparnica02}
{Oparnica, Lj.}
\newblock Generalized fractional calculus with applications in mechanics.
\newblock {\em Matemati\v cki vesnik}, {\bf 53}:151--158, 2002.

\bibitem{Wloka}
{Wloka, J.}
\newblock {\em Partial differential equations}.
\newblock Cambridge University Press, Cambridge, second edition, 1992.

\bibitem{YavariSarkani}
{Yavari, A., Sarkani, S.}
\newblock On applications of generalized functions to the analysis of
  {E}uler-{B}ernuli beam columns with jump discontinuities.
\newblock {\em Int. J. Mech. Sci.}, {\bf 43}(6):1543--1562, 2001.

\bibitem{YavariSarkaniReddy}
{Yavari, A., Sarkani, S., Reddy, J. N.}
\newblock On nonuniform {E}uler-{B}ernoulli and {T}imoshenko beams with jump
  discontinuities: application of distribution theory.
\newblock {\em Int. J. Solid. Struct.}, {\bf 38}(46-47):8389--8406, 2001.

\end{thebibliography}

 \end{document}